\documentclass[11pt,twoside]{amsart}
\usepackage[latin1]{inputenc}
\usepackage[T1]{fontenc}
\usepackage{mathtools}
\usepackage{graphicx}
\usepackage{tikz}
\usetikzlibrary{chains}
\usepackage{tcolorbox}
\usepackage{mathabx}
\usepackage{enumerate} 

\PassOptionsToPackage{pdfusetitle,pagebackref,colorlinks}{hyperref}
\usepackage{bookmark}
\hypersetup{
  linkcolor={red!70!black},
  citecolor={green!70!black},
  urlcolor={blue!80!black}
}

\usepackage[latin1]{inputenc}
\usepackage{amsmath}
\usepackage{amsthm}
\usepackage{amssymb}

\usepackage[all]{xy}
\usepackage{hyperref}
\newtheorem{thm}{Theorem}[section]

\newtheorem{prop}[thm]{Proposition}

\newtheorem{lem}[thm]{Lemma}
\newtheorem{cor}[thm]{Corollary}

\numberwithin{equation}{section}

\theoremstyle{definition}

\newtheorem{remark}[thm]{Remark}

\usepackage[OT2,T1]{fontenc}
\DeclareSymbolFont{cyrletters}{OT2}{wncyr}{m}{n}
\DeclareMathSymbol{\Sha}{\mathalpha}{cyrletters}{"58}

\newcommand{\CH}{{\rm CH}}

\newcommand{\Pic}{{\rm Pic}}
\newcommand{\im}{{\rm Im}}
\newcommand{\rk}{{\rm rk}}

\newcommand{\cal}{\mathcal}

\newcommand{\kk}{{\cal K}}
\newcommand{\kl}{{\cal L}}

\newcommand{\km}{{\cal M}}
\newcommand{\ko}{{\cal O}}

\newcommand{\ks}{{\cal S}}
\newcommand{\kt}{{\cal T}}

\newcommand{\LL}{\mathbb{L}}
\newcommand{\ZZ}{\mathbb{Z}}
\newcommand{\QQ}{\mathbb{Q}}

\newcommand{\CC}{\mathbb{C}}

\newcommand{\PP}{\mathbb{P}}

\newcommand{\hh}{\mathfrak{h}}

\renewcommand{\to}{\xymatrix@1@=15pt{\ar[r]&}}
\renewcommand{\rightarrow}{\xymatrix@1@=15pt{\ar[r]&}}
\renewcommand{\leftarrow}{\xymatrix@1@=15pt{&\ar[l]}}
\renewcommand{\mapsto}{\xymatrix@1@=15pt{\ar@{|->}[r]&}}
\renewcommand{\twoheadrightarrow}{\xymatrix@1@=18pt{\ar@{->>}[r]&}}
\renewcommand{\hookrightarrow}{\xymatrix@1@=15pt{\ar@{^(->}[r]&}}
\newcommand{\hook}{\xymatrix@1@=15pt{\ar@{^(->}[r]&}}
\newcommand{\congpf}{\xymatrix@L=0.6ex@1@=15pt{\ar[r]^-\sim&}}
\renewcommand{\cong}{\simeq}

\def\E{{\mathcal E}}

\usepackage{comment}
\excludecomment{comment}
\begin{document}

\title[]{Nodal quintic surfaces  and lines on cubic fourfolds}

\author[D.\ Huybrechts]{D.\ Huybrechts\\
with an appendix by J.\ Ottem}

\address{DH: Mathematisches Institut and Hausdorff Center for Mathematics,
Universit{\"a}t Bonn, Endenicher Allee 60, 53115 Bonn, Germany}
\email{huybrech@math.uni-bonn.de}
\address{JO: Department of Mathematics, University of Oslo, Box 1053, Blindern, 0316 Oslo, Norway}
\email{johnco@math.uio.no}

\begin{abstract} 

 We study nodal quintic surfaces with an even set of $16$ nodes as analogues of singular Kummer surfaces. 
 The interpretation of the natural double cover of an even $16$-nodal quintic as a certain Fano variety of lines could be viewed as a replacement for the additive structure of the cover of a singular Kummer surface by its associated  abelian surface.
 
Most of the results in this article can be seen as refinements of known facts and our arguments rely heavily on techniques developed by Beauville \cite{Beamu5}, Murre \cite{Murre}, and Voisin \cite{VoisinGT}. Results due to 
 Shen \cite{Shen0,Shen2} are particularly close to some of the statements. In this sense, the text is mostly expository (but with complete proofs), although our arguments often differ substantially from the original sources.

\end{abstract}

\maketitle

{\let\thefootnote\relax\footnotetext{The  author is supported by the 
ERC Synergy Grant HyperK (Grant agreement No. 854361) and the Hausdorff Center for Mathematics.}}
A nodal quartic surface $S\subset\PP^3$ can have at most $16$ nodes. When the maximum is attained, the minimal resolution $\tau\colon\tilde S\to S$ is a K3 surface and, more precisely, a Kummer surface. In other words, $S$ is then a singular Kummer surface, so it is isomorphic to the quotient $A/\iota$ of an abelian surface $A$ by its natural involution $\iota$.
Ultimately, the link to abelian surfaces relies on the fact that the set $\{E_i\}$ of the $16$ exceptional lines of the resolution
$\tau$ is even, i.e.\ $\sum E_i=0$ in $H^2(\tilde S,\ZZ/2\ZZ)$.
Then the double cover $\tilde A\to \tilde S$ ramified along $\bigcup E_i\subset\tilde S$ is the blow-up of $A$ in its two-torsion points. The situation was first studied by Kummer and marks the beginning of the theory of K3 surfaces. The situation for nodal quadrics and nodal cubics is also well understood. For example,
a nodal cubic can have at most four nodes and in this case it is the Cayley cubic.
\smallskip

The maximal number of nodes of a nodal quintic surface $D\subset \PP^3$ was determined by Beauville \cite{Beamu5}. It is $31$, which is realized by the Togliatti surface. However, in this case 
the set $\{E_i\}$ is not even. As Beauville showed  \cite[\S2]{Beamu5}, only a set of $16$ or of $20$ nodes can lead to an even
set of exceptional lines $\{E_i\}$. The two situations have been studied by Beauville \cite{Beauvcan} and Catanese \cite{Catanese}. In this article, we are concerned with even
$16$-nodal quintic surfaces, i.e.\ nodal quintic surfaces with exactly $16$ nodes such that 
the integral cohomology class $\sum E_i$ on the minimal resolution is divisible by two. In this case, $D\cong F/\iota$, where $F$ is a smooth surface
of general type endowed with an involution $\iota$ with $16$ fixed points. Its blow-up in the
fixed points can be viewed as the double cover $\tilde F\to \tilde D$ of the minimal resolution
$\tilde D\to D$ ramified along the union $\bigcup E_i\subset\tilde D$ of all exceptional lines. In this sense, $D$ is the analogue of a singular Kummer surface and $F$ is the quintic analogue of the abelian surface $A$. It is now tempting to try to develop a theory that parallels the classical  theory of abelian surfaces and their associated (singular) Kummer surfaces. But how far can one push this analogy?  The following picture already includes the objects that we will study.

\begin{picture}(200,230)
\put(63,205){$\xymatrix@C=18pt@R=18pt{ \tilde A\ar[d]\ar[r]&\tilde S\ar@<-3ex>[d]\phantom{DDD}\\
\iota \circlearrowright A\phantom{\iota\circlearrowleft~}\ar<1.5ex,-7.5ex>;[r]!<10ex,0ex>&S\subset \PP^3}$}
\put(63,125){$\xymatrix@C=18pt@R=18pt{\tilde F\ar[d]\ar[r]&\tilde D\ar@<-3ex>[d]\phantom{DDD}\\
\iota \circlearrowright F\phantom{\iota\circlearrowleft~}\ar<1.5ex,-7.5ex>;[r]!<10ex,0ex>\ar[r]&D\subset \PP^3}$}
\put(230,208){$H^\ast(A,\ZZ)^-= H^1(A,\ZZ)\oplus H^3(A,\ZZ)$}
\put(230,188){$\CH(A)^-= A\times \hat A$}
\put(230,168){$\hh(A)^-= \hh^1(A)\oplus\hh^3(A)$}

\put(230,128){$H^\ast(F,\ZZ)^-= H^2(F,\ZZ)^-$}
\put(230,108){$\CH(F)^-= \CH_0(F)^-\oplus \CH_1(F)^-$}
\put(230,88){$\hh(F)^-= \hh^2(F)^-$}
\end{picture}
\vskip-2.8cm
The first step is to find the correct analogue of the additional geometric structure
of the double cover $A\to S$ provided by the interpretation of $A$ as an abelian surface.
Already the work of Fano \cite{Fano} and Togliatti \cite{Togliatti} suggests where to look.
For any smooth cubic fourfold $X\subset\PP^5$ the variety of lines
$L'\subset X$ intersecting a fixed generic line $L\subset X$ is a smooth surface $F_L$
endowed with a natural involution $\iota$ mapping $L'$ to the residual line of $L\cup L'\subset
X\cap \overline{LL'}$.  By mapping $L'\in F_L$ to the point of intersection of
the plane $\PP^2\cong \overline{LL'}$ with a generic $\PP^3\subset\PP^5$, the quotient $D_L\coloneqq F_L/\iota$ is realized as an even $16$-nodal quintic surface with the nodes corresponding to the fixed points of $\iota$. From this perspective, the surfaces $F$ and $D$ have been studied by Voisin \cite{VoisinGT}. Understanding their geometry is a crucial step in her proof of the global Torelli theorem for cubic fourfolds. Viewing the double cover $F\to D$ of an even $16$-nodal quintic as $F\cong F_L$ can be seen as the analogue of viewing the double cover $A\to S$ as an abelian surface. \smallskip

There are clear limitations to the analogy between the double covers $A\to S$ of a singular Kummer quartic surface and the natural double covers $F\to D$ of  an even $16$-nodal quintic surface. For one, $F$ is a surface of general type with ample canonical bundle $\omega_F$ while the abelian surface $A$ has of course trivial canonical bundle. Also, from a Hodge theoretic perspective, $A$ and $F$ 
are  quite different. For example, the Hodge structure $H^1(F,\ZZ)$ of weight one is trivial, although 
the minimal resolutions $\tilde S$ and $\tilde D$ are both simply connected. Although this phenomenon is well known,
a clear topological reason for the different behaviour of the corresponding double covers $\tilde A\to\tilde S$ and
$\tilde F\to \tilde D$ seems missing. 
\smallskip

This article naturally splits in two interwoven parts: The geometry of even $16$-nodal quintic surfaces
$D\subset\PP^3$ together with their double covers $F\to D$ and the geometry
of Fano varieties of lines on cubic fourfolds.

\subsection{}
In the first section of this article we collect the known facts about even $16$-nodal quintic surfaces 
$D\subset \PP^3$ and their natural double cover $F\to D$. 
We summarize the information about $F$ in the following theorem. Note that (i) is due to Voisin
 \cite{VoisinGT} and (iii) strengthens a result of Beauville  \cite{Beamu5}. I presume that the other assertions are also more or less well known to the experts.
 
\begin{thm}\label{thm:topoDF}
Let $D\subset \PP^3$ be an even $16$-nodal quintic surface written as a quotient $D\cong F/\iota$.\smallskip

\noindent
{\rm (i)} The numerical invariants of $F$ are as follows:  
$$\chi(F,\ko_{F})=6, ~q(F)=0,~p_g(F)=5,~e(F)=62,~b_2(F)=60,
\text{ and }{\rm c}_1^2(F)=10.$$

\noindent
{\rm (ii)}  The surface $F$ is algebraically simply connected.\footnote{In the appendix by J.\ Ottem, the Fano perspective will be used to show that $F$ is also topologically simply connected.}

\noindent
{\rm (iii)}  The integral cohomology $H^2(F,\ZZ)$ is torsion free.

\noindent
{\rm (iv)}  The anti-invariant part $H^2(F,\ZZ)^-$ is a K3 Hodge structure, i.e.\  a Hodge structure 
of weight two with a one-dimensional $(2,0)$-part.

\noindent
{\rm (v)}  The intersection form on $H^2(F,\ZZ)^-$ is even and definite of signature $(2,21)$.
\end{thm}

For the full cohomology one finds $H^\ast(F,\ZZ)^-=H^2(F,\ZZ)^-$, while for
a singular Kummer surfaces and the covering abelian surface we have
$H^\ast(A,\ZZ)^-=H^1(A,\ZZ)\oplus H^3(A,\ZZ)$.

\subsection{} Even $16$-nodal quintics naturally occur in the study of smooth
cubic fourfolds $X\subset\PP^5$. Recall that the Fano variety $F(X)$ of lines contained in $X$ is
a hyperk\"ahler fourfold \cite{BD}.  For a generic line $L\subset X$ the variety $F_L\subset F(X)$ of
all lines intersecting $L$ is a smooth surface with a natural involution $\iota$. Its quotient 
$D_L\coloneqq F_L/\iota$ is an even $16$-nodal quintic.
\smallskip

Due to a classical result of Beauville and Donagi \cite{BD},
there exists a Hodge isometry $(H^4(X,\ZZ)_{\rm pr}(1),-(~.~))\cong (H^2(F(X),\ZZ)_{\rm pr},q)$, where $q$ is the Beauville--Bogomolov--Fujiki pairing restricted to the Pl\"ucker primitive cohomology.
The next theorem complements this result by another Hodge isometry which refines results due to  Izadi  \cite{Izadi} and Shen \cite{Shen0,Shen2}, see  Section \ref{sec:IzadiShen} for a detailed comparison.

\begin{thm}\label{thm:FLF}
Let $L\in F(X)$ be a generic line contained in a smooth cubic fourfold $X\subset\PP^5$ and let
$F_L\subset F(X)$ be the smooth surface of all lines intersecting $L$. Then there exist natural
isometries of Hodge structures
$$(H^4(X,\ZZ)_{\rm pr}(1),-(~.~))\cong (H^2(F(X),\ZZ)_{\rm pr},q)\cong (H^2(F_L,\ZZ)_{\rm pr}^-,(1/2)(~.~)).$$
\end{thm}
On the right hand side, $H^2(F_L,\ZZ)_{\rm pr}$ denotes the primitive
part with respect to the restriction of the Pl\"ucker polarization 
and $H^2(F_L,\ZZ)_{\rm pr}^-\subset H^2(F_L,\ZZ)_{\rm pr}$ is its $\iota$-anti-invariant part.
The standard intersection form $(~.~)$ on the surface $F_L$ is scaled by the factor $(1/2)$

\smallskip

The following geometric global Torelli type result is an analogue of a well-known result for cubic threefolds, for a more complete version  see Section \ref{sec:spellout}.

\begin{cor}\label{cor:GT1}
Assume $X,X'\subset \PP^5$ are two smooth cubic fourfolds and let
$L\subset X$ and $L'\subset X'$ be two generic lines. Then $X\cong X'$ if and only if
there exists an isometry of Hodge structures $$(H^2(F_L,\ZZ)^-_{\rm pr},(~.~))\cong (H^2(F_{L'},\ZZ)^-_{\rm pr},(~.~)).$$ 
\end{cor}

The assumptions on $L$ and $L'$ simply mean that the two Fano varieties
 $F_L\subset F(X)$ and $F_{L'}\subset F(X')$ of lines intersecting $L\subset X$
and $L'\subset X'$ are smooth surfaces.\smallskip

A more careful analysis of the situation allows one to upgrade  the Hodge isometry in
Theorem \ref{thm:FLF} to the level of Chow groups and rational Chow motives. 
The first assertion of the next theorem was first proved by Shen, see \cite[Thm.\ 1.5]{Shen2} and \cite[Thm.\ 4.7]{Shen0}. The second part, a direct consequence of the first, formulated in terms of transcendental motives, is due to Bolognesi and Pedrini \cite[Prop.\ 2.7]{BP}. See Section \ref{sec:ShenBP} for a comparison of the techniques.

\begin{thm}\label{thm:mainChow}\label{thm:mainmot}
The Fano correspondence induces an isomorphism
$$ \CH_0(F_L)_{\rm hom}^-\cong \CH_1(X)_{\rm hom}$$
of the homologically trivial integral Chow groups and  an isomorphism of rational Chow motives
$$\hh^2(F_L)_{\rm pr}^-\cong \hh^4(X)_{\rm pr}(1).$$
\end{thm} 

To compare the result with the case of Kummer surfaces, observe that $\hh^2(F_L)^-$ is 
the quintic analogue of the motive $\hh^1(A)\oplus \hh^3(A)$ of an abelian surface

\subsection{} Here is an outline of the content. Section \ref{sec:Top} deals with the topological and Hodge theoretic invariants of an even $16$-nodal quintic and its natural double cover. 
It contains the proofs of the various parts of Theorem \ref{thm:topoDF}. With one minor exception in the proof of Lemma \ref{lem:antiinvH} and for convenience reasons only, we do not resort to the interpretation in terms of Fano varieties of lines on cubic fourfolds.

  Section \ref{sec:Fanopersp} discusses the natural occurrence of
double covers of even $16$-nodal quintics as  Fano varieties of lines contained in a cubic fourfold intersecting a fixed line. This interpretation, allows one to consider the double cover endowed
with the restriction of the Pl\"ucker polarization which is not canonical from the point of view of the quintic. This leads to the K3 Hodge structure of rank $22$
that appears in Theorem \ref{thm:FLF}, the proof of which is postponed to the short Section \ref{sec:proof}. 

Chow groups with integral coefficients and rational Chow motives are discussed in Section \ref{sec:Chow}. The proofs of the two parts of Theorem \ref{thm:mainmot} can be found in Sections \ref{sec:proofCH} and \ref{sec:proofMot}. The appendix contains the proof of the simple connectedness of $F_L$.

\subsection{} As should be clear from the introduction, the paper builds heavily on earlier work of Beauville, Murre, Voisin, and others. Many of the results can be found in a similar form in the literature, especially in the work of  Shen \cite{Shen0,Shen2}, cf.\  Section \ref{sec:IzadiShen}. 
 It seemed worthwhile to add further observations and refinements to their results and to present a coherent and streamlined picture of what is currently known. The paper was originally intended to become part of \cite{HuyCubics}, but it grew out of proportion and it seemed more appropriate to publish it separately.

\medskip

\noindent{\bf Acknowledgements: }
I wish to thank Rita Pardini, Peter Teichner, and Gerard van der Geer for email correspondences
and help with the literature. Comments of Fabrizio Catanese on the first version are gratefully acknowledged.
Special thanks to Alexander Kuznetsov for a long list of insightful and constructive comments and suggestions.

\section{The topology of an even $16$-nodal quintic and its double cover}\label{sec:Top}

In the following $D\subset\PP^3$ will always denote a quintic surface with $16$
nodes $x_1,\ldots,x_{16}\in D$ as its only singularities. By the adjunction formula,
its canonical bundle  is $\omega_D\cong\ko_D(1)$.

\subsection{} The minimal resolution $\tau\colon \tilde D\to D$ with canonical
bundle $\omega_{\tilde D}\cong\tau^\ast\ko_D(1)$ replaces the nodes $x_i$
by the $(-2)$-curves $E_i\subset\tilde D$.
The surface $D$  is called an \emph{even $16$-nodal quintic} if in addition $\sum E_i=0$ in $H^2(\tilde D,\ZZ/2\ZZ)$ or, equivalently, if the line bundle $\ko(\sum E_i)$ admits a  square root. Note that the line bundle $\kl$ with $\kl^2\cong\ko(\sum E_i)$ is unique, as $\Pic(\tilde D)$ is torsion free due to $\pi_1(\tilde D)=\{1\}$, see below.
 Unlike the case of $16$-nodal quartics, cf.\ \cite[Rem.\ 14.3.19]{HuyK3}, it seems unclear whether a $16$-nodal quintic is automatically even.\footnote{F.\ Catanese informed the author that there are $16$-nodal quintics which are not even and that this follows from a forthcoming paper of his.}

Similar to the well-known smoothing of ordinary double points on a singular quartic 
Kummer surface, the surface $\tilde D$ is diffeomorphic to a smooth quintic surface, cf.\  \cite[\S 2]{Beamu5} or for more general results \cite{Man}. (Ideally, one would like $\tilde D$ to be deformation equivalent to a smooth quintic, which would need the nodes to be independent \cite{BW}.)
In particular, by the Lefschetz theorem, $\tilde D$ is simply connected with
$$b_2(\tilde D)=53\text{ and } e(\tilde D)=55.$$
Hence, for the $16$-nodal quintic $D$ one has
$$b_2(D)=37\text{ and } e(D)=39.$$

\subsection{} Assume now that $D$ is an even $16$-nodal quintic and let  $\tilde\pi\colon\tilde F\to \tilde D$ be the double 
cover of $\tilde D$ ramified along
$\bigcup E_i\subset\tilde D$, i.e.\ $\tilde\pi$ is the
cyclic double cover associated with the line bundle $\kl$ and the section of $\kl^2$
defining $\bigcup E_i$. We will use the same notation for the reduced pre-images of the $(-2)$-curves  and write $E_i\subset\tilde F$. These $(-1)$-curves can also be viewed as the fixed components of the natural involution $\tilde \iota$ of $\tilde F$. Then $\omega_{\tilde F}\cong\ko(\sum E_i) \otimes \tilde\pi^\ast\omega_{\tilde D}\cong
\ko(\sum E_i) \otimes \tilde\pi^\ast\tau^\ast\ko_D(1)$. The blow-down $\tilde\tau\colon \tilde F\to F$ of the $(-1)$-curves $E_i\subset \tilde F$ fits into the commutative diagram
$$\xymatrix{\tilde F\ar[d]_-{\tilde\tau}\ar[r]^-{\tilde \pi}&\tilde D\ar[d]^-\tau\\
F\ar[r]_-\pi&D.}$$
The smooth surface $F$ with its canonical bundle
$\omega_F\cong\pi^\ast\omega_D\cong\pi^\ast\ko_D(1)$ comes with a covering involution
$\iota$ with $16$ fixed points over the nodes $x_i$ and its quotient is $F/\iota\cong D$. Throughout, we will write $x_i$ for the fixed points in $F$ and
for their images in $D$.
\smallskip

The numerical invariants of the surface $F$ are as follows:
$$\chi(F,\ko_{F})=6, ~q(F)=0,~p_g(F)=5,~e(F)=62,~b_2(F)=60,
\text{ and }{\rm c}_1^2(F)=10.$$

These facts are either classical \cite{Fano,Togliatti} or have been proved by Voisin \cite{VoisinGT},
see also \cite[Sec.\ 6.4.5]{HuyCubics} for a variant of the proof. The most surprising
fact is probably the regularity of the surface $F$, which can equivalently be phrased
as $H^1(F,\ZZ)=0$. This is in stark contrast to $H^1(A,\ZZ)\ne0$ for the natural double cover
of a $16$-nodal quartic surface. To prove this, Voisin \cite[\S 3, Lem.\ 3]{VoisinGT} uses the Fano description of $F$, cf.\  Remark \ref{rem:Xiao} below or \cite[Sec.\ 6.4.5]{HuyCubics} for a more classical argument.

The fact that $F$ is regular in particular shows $$\Pic(F)\subset H^2(F,\ZZ)\text{ and } {\rm Tors}\,\Pic(F)={\rm Tors}\, H^2(F,\ZZ).$$

\subsection{}\label{sec:decompH2}
The action of the involution $\iota$ on $F$ leads to an eigenspace decomposition
$$H^0(F,\omega_{F})=H^0(F,\omega_{F})^+\oplus H^0(F,\omega_{F})^-$$
into a four-dimensional invariant part $ H^0(D,\omega_{D})\cong H^0(F,\omega_{F})^+$ and a one-dimensional anti-invariant part $ H^0(F,\omega_{F})^-$.\smallskip

\begin{proof} Indeed, as  $\omega_{D}\cong\ko_D(1)$ for the quintic $D$, we have $h^0(D,\omega_{D})=4$. Since $F\twoheadrightarrow D$ is an \'etale double cover over the complement of the nodes  $x_i$, we know
 $\pi^*\colon H^0(D,\omega_{D})\congpf H^0(F,\omega_{F})^+$. Eventually, use $h^0(F,\omega_{F})=p_g(F)=5$. 
 \end{proof}
 
 Clearly, the composition $\pi\colon F\to D\subset\PP^3$ is the morphism associated
 with the invariant linear system $|\omega_F|^+\subset |\omega_F|$.

\subsection{}\label{sec:invpart}
The invariant and anti-invariant parts $H^2(F,\ZZ)^\pm\subset H^2(F,\ZZ)$ of the induced action of the involution $\iota$ satisfy
$$\rk\, H^2(F,\ZZ)^+=37\text{ and } \rk\, H^2(F,\ZZ)^-=23.$$
In the following, $H^2(F,\ZZ)^-$ is considered as a Hodge structure of K3 type (as we will see
torsion free). 
 Furthermore, the intersection pairing on $H^2(F,\ZZ)^-$ has signature $(2,21)$.
\smallskip

\begin{proof} In order to compute the ranks, we can work with rational cohomology.
Then $H^2(D,\QQ)\cong H^2(F,\QQ)^+$ by a
standard spectral sequence argument \cite[Prop.\ 5.2.3]{GrothToh}. 
Combining $e(D)=39$ with $H^1(D,\QQ)\cong H^1(F,\QQ)=0$ to deduce
$\rk\,H^2(F,\ZZ)^+=b_2(D)=37$, we can conclude by
$\rk\, H^2(F,\ZZ)^-=b_2(F)-37=23$. Alternatively,\footnote{Thanks to A.\ Kuznetsov for suggesting this.} one can combine the Lefschetz fixed point formula $16=2+b_2^+-b_2^-$ with $62=e(F)=2+b_2=2+b_2^++b_2^-$ to compute the 
dimensions of $H^2(F,\QQ)^\pm$.

The last assertion follows
from the orthogonal decomposition $H^2(F,\QQ)=H^2(F,\QQ)^+\oplus H^2(F,\QQ)^-$ and the Hodge index theorem on $H^2(D,\QQ)\cong H^2(F,\QQ)^+$.
\end{proof}

Note that the result is in contrast to  $H^2(A,\ZZ)=H^2(A,\ZZ)^+$  and $H^2(A,\ZZ)^-=0$ for the natural involution of an abelian surface. However, $H^*(A,\ZZ)^-=H^1(A,\ZZ)\oplus H^3(A,\ZZ)$.

\subsection{}\label{sec:no2tor}
It is known that not only $\tilde D$ but also the nodal quintic $D$ is simply connected, cf.\ \cite[Sec.\ 5.2]{Dimca}.
In particular, $H_1(D,\ZZ)=0$ and by the universal coefficient theorem $H^2(D,\ZZ)$ is torsion free.
The same holds for $F$, see also Remark \ref{rem:Xiao}.

\begin{lem}\label{lem:tfH2}
The integral cohomology $H^2(F,\ZZ)$ is torsion free.
\end{lem}

\begin{proof} We first show that $H^2(F,\ZZ)$ has no two-torsion, i.e.\ $H^2(F,\ZZ)[2]=0$
or, equivalently, $\Pic(F)[2]=0$. This is an immediate consequence of 
results of Beauville. According to \cite[Lem.\ 2]{Beamu5} one has
 $$\Pic(F)[2]\cong\ker\left((\ZZ/2\ZZ)^{\oplus 16}\to \Pic(\tilde D)\otimes\ZZ/2\ZZ\right)/\,\ZZ e.$$
 
Here, $e=(1,\ldots,1)$ and the map is given by the classes $\ko(E_i)\in\Pic(\tilde D)$.
Furthermore, by virtue of \cite[Prop.\ \S2]{Beamu5}, if a collection $E_i$, $i\in I$,
is even, i.e.\ $\ko(\sum_{i\in I} E_i)$ has a square root, then $|I|=16$ or $20$. Hence,
$\Pic(F)[2]=0$.
\smallskip

To conclude one can evoke a result by Ciliberto and Mendes Lopes \cite[Thm.\ A]{CilML}:
The torsion of $\Pic(X)$ of a regular, minimal surface of general type with ${\rm c}_1^2(X)=2\chi(\ko_X)-2$ is either trivial or $\ZZ/2\ZZ$. The assumptions are satisfied for
the surface $F$. For the reader's convenience and later use, we recall the key argument.
First, one shows $|\pi_1^{\rm alg}|\leq 3$ by the following argument originally due to Bombieri:
Assume $F_0\to F$ is an irreducible \'etale cover of degree $d$, then $${\rm c}_1^2(F_0)=d\cdot{\rm c}_1^2(F)=d\cdot(2\chi(\ko_F)-2)=d\cdot((2/d)\cdot \chi(\ko_{F_0})-2)=2\chi(\ko_{F_0})-2d.$$ Noether inequality $2p_g(F_0)\leq {\rm c}_1^2(F_0)+4$ together with the obvious $2\chi(\ko_{F_0})-2\leq 2 p_g(F_0)$ then
give ${\rm c}_1^2(F_0)+2(d-1)\leq {\rm c}_1^2(F_0)+4$, which proves $d\leq 3$. 

The case $d=3$
is excluded by  \cite[(1.2)]{CilML}. Alternatively, one can observe that for $d=3$ the two inequalities above are equalities. Hence, the surface $F_0$ is on the Noether line with $q(F_0)=0$, $p_g=17$, and
${\rm c}_1^2(F_0)=30$. According to results of Horikawa, see \cite[Sec.\ VII, 9]{BPV}, the minimal resolution of the canonical model of $F_0$ is then a Hirzebruch surface. As the canonical model
of $F_0$ dominates the canonical model of $F$ which in turn maps onto $D$, this results in a contradiction.
\end{proof}

\begin{lem}\label{lem:simple}
The surface $F$ is algebraically simply connected, i.e.\ $\pi_1^{\rm alg}(F)=\{1\}$.
\end{lem}

\begin{proof}
According to \cite[Cor.\ 4.4]{Xiao}, every minimal surface of general type satisfying 
${\rm c}_1^2<(8/3)(\chi(\ko)-2)$ is algebraically simply connected. As the
inequality holds for the surface $F$, one indeed has $\pi_1^{\rm alg}(F)=\{1\}$.

Alternatively, one may use the arguments in the proof of Lemma \ref{lem:tfH2}. There we
saw that irreducible \'etale covers $F_0\to F$ are of degree at most two. However, an \'etale double cover corresponds to a line bundle of order two and thus defines a non-trivial two-torsion class in 
$H^2(F,\ZZ)$, the existence of which is excluded by the previous lemma
or rather by its first step showing $\Pic(F)[2])=0$.
\end{proof}

The following immediate consequence is originally due to Shen
\cite[Lem.\ 4.5 \& Sec.\ 5]{Shen0}. The original proof is rather involved and makes heavy use of the Fano interpretation of $F$.

\begin{cor}\label{cor:simple}
The surface $F$ satisfies $H_1(F,\ZZ)=0$.\qed
\end{cor}

\begin{remark}\label{rem:Xiao}
(i) Lemmas \ref{lem:tfH2} and \ref{lem:simple} are essentially equivalent. Clearly,
Lemma \ref{lem:simple} and its consequence Corollary \ref{cor:simple} combined
with the universal coefficient theorem imply that $H^2(F,\ZZ)$ is torsion free.
Conversely, the torsion freeness was used to show that $F$ is algebraically simply connected.
Also note that $H_1(F,\ZZ)=0$ implies
the regularity of $F$, which gives an alternative argument to Voisin's original proof \cite{VoisinGT}.
\smallskip

(ii) From the above, it is not clear whether also the topological fundamental group
$\pi_1(F)$ is trivial.\footnote{Now proved in the appendix.} 
Since $F'\coloneqq F\setminus\{x_i\}\to D'\coloneq D\setminus\{x_i\}$ is an \'etale cover, $\pi_1(F)\cong \pi_1(F')$ is the kernel of the surjection $\pi_1(D')\to\ZZ/2\ZZ$ mapping a simple loop $\gamma_i$ in $D$ around $x_i\in D$ (or, equivalently, in $\tilde D$ around $E_i\subset\tilde D$) to the generator. Here, the fundamental group $\pi_1(D')$ is viewed as a quotient of $(\ZZ/2\ZZ)^{\ast 16}$. 
Since by the above the pro-finite completion $\pi_1^{\rm alg}(F)$ of
$\pi_1(F)$ is trivial, $F$ is simply connected if and only if its fundamental group $\pi_1(F)$ is residually finite.
\end{remark}

\begin{remark} A nodal quintic $D\subset \PP^3$ with an even set of $20$ nodes gives rise
to a double cover $F\to D$ satisfying $q(F)=0$, $p_g(F)=4$, and ${\rm c}_1^2(F)=10$. In particular, the composition $F\to\PP^3$ is the canonical map. Examples were first constructed
by Beauville \cite[Prop.\ 3.6]{Beauvcan} and a construction due to Gallarati was studied in detail by Catanese \cite{Catanese}. A special example was described by van der Geer and Zagier 
\cite[\S 4]{VdGZ} as the minimal model of a Hilbert modular surface associated with $\QQ(\sqrt{21})$.

It turns out that quintics with an even set of $20$ nodes form an irreducible family \cite[Rem.\ 3.7, (2)]{Beauvcan}. However, without the condition on the $20$ nodes to be even the moduli space is reducible \cite[Prop.\ 3.8]{Catanese}.
The Hilbert modular surface is simply connected \cite[Thm.\ 6.1]{vdG}, hence for any quintic with an even set of $20$ nodes the double cover $F$ is simply connected.
\end{remark}

\subsection{}\label{sec:PicDD} The Picard group $\Pic(D')$ of the open subset
$D'=\tilde D\setminus\bigcup E_i=D\setminus\{\bar x_i\}$ is generated by 
the image of the restriction map $\Pic(D)\to \Pic(D')$ and
the two-torsion line bundle $\kl|_{D'}$ corresponding to
$F'\twoheadrightarrow D'$.
\smallskip

\begin{proof}
Indeed, the restriction map $\Pic(\tilde D)\to \Pic(D')$ is surjective and its kernel is
generated by the line bundles $\ko(E_i)$. Furthermore, $\Pic(\tilde D)$
is generated by $\Pic(D)$, the line bundles $\ko(E_i)$ and all
line bundles $\km$ with $\km^2\cong\ko(\sum_{i\in I}\ko(E_i))$ for some $I\subset\{1,\ldots,16\}$.
However, the kernel of the pull-back $\Pic(D')\to \Pic(F')$ is generated by $\kl|_{D'}$, see for example \cite[Lem.\ 3.1]{OP}. Thus, if $|I|<16$, such an $\km$ would pull-back
to a non-trivial two-torsion line bundle on $F'$. Since $\Pic(F')\cong\Pic(F)$, this would contradict Lemma \ref{lem:tfH2}.
\end{proof}

\subsection{}\label{sec:antinv}
The following assertion is analogous to the classical fact that for any \'etale double
cover $C\to \bar C$ of a smooth curve the image of the
map $1-\iota^\ast\colon H^1(C,\ZZ)\to H^1(C,\ZZ)$ is the anti-invariant part
$H^1(C,\ZZ)^-\subset H^1(C,\ZZ)$, see \cite[Sec.\ 5.3.2]{HuyCubics} for references.

\begin{lem}\label{lem:antiinvH}
The image of the map 
\begin{equation}\label{eqn:1iota}
1-\iota^*\colon H^2(F,\ZZ)\to H^2(F,\ZZ)
\end{equation} is
the anti-invariant part $H^2(F,\ZZ)^-=\{\alpha\in H^2(F,\ZZ)\mid \iota^*\alpha=-\alpha\}$.
\end{lem}

\begin{proof} 
Clearly, the image of (\ref{eqn:1iota}) is contained in the anti-invariant part. The difficult part
is to show that every anti-invariant class $\alpha\in H^2(F,\ZZ)^-$ is of the form $\alpha=\beta-\iota^\ast\beta$ for some integral class $\beta\in H^2(F,\ZZ)$.

We first consider the class $\alpha=[E]$ of a divisor $E$ on $F$, which typically is not effective
nor irreducible. In this case, one can use the following argument, which is inspired by \cite[Lem.\ 0.3.4]{BeauvPrym}, see also \cite[Lem.\ 3.1]{OP}.  By using the projectivity of $F$, we may assume that 
at least one of the fixed points $x_1,\ldots,x_{16}\in F$ is not contained in the support
of $E$. Then 
$\iota^\ast E+E$ is a principal divisor $(f)$ for some $f\in K(F)$. Hence,
$(\iota^\ast f)=\iota^\ast(\iota^\ast E+E)=E+\iota^\ast E=(f)$ and, therefore, $\iota^\ast f=\lambda\cdot f$ for some $\lambda\in\CC^\ast$. By evaluating at one point
$x_i$ not contained in the support of $E$, one finds $\lambda=1$, i.e.\ $\iota^\ast f=f$
or, equivalently, $f=\pi^*g$ for some $g\in K(D)$. Now write $E=E_1-E_2+\pi^{-1}(E_0)$, where $E_1,E_2\subset F$
are effective divisors without common irreducible components and
 $\pi^{-1}(E_0)$ is the $\iota$-invariant part  of $E$. In other words, 
if $E_i'\subset E_i$, $i=1,2$, is an irreducible component then $\iota^*E_i'$ is not an irreducible component of $E_i$.
 
Then $\iota^\ast E+E=(\pi^*g)$ can be rewritten as 
$\iota^\ast E_1+E_1=\iota^\ast E_2+E_2+\pi^{-1}(E_0')$, where $E_0'=(g)-2E_0$.
By our assumptions on $E_1$ and $E_2$, this implies $E_1=\iota^\ast E_2$ or, equivalently, $E_2=\iota^\ast E_1$. Therefore, $E=E_1-\iota^\ast E_1+\pi^{-1}(E_0)$,
which passing to cohomology becomes $\alpha=[E]=[E_1]-\iota^\ast[E_1]+\pi^*[E_0]$. Since the last summand is contained in the invariant part, it has to be trivial, which proves
$\alpha=(1-\iota^\ast)\beta$ for $\beta=[E_1]$.

As the assertion is purely topological, it is invariant under deformations. It thus suffices to argue that $H^2(F,\ZZ)^-$ is generated by classes that are of type $(1,1)$ on some
deformation. This can be done directly or by using the Fano perspective. We defer the proof to
Remark \ref{rem:deferal}.
\end{proof}

\begin{cor}\label{cor:evenInt}
For all classes $\alpha_1,\alpha_2\in H^2(F,\ZZ)^-$ the intersection pairing
$(\alpha_1.\alpha_2)$ is even.
\end{cor}

\begin{proof}
Indeed, writing  $\alpha_i=\beta_i-\iota^\ast\beta_i$, 
$i=1,2$, with $\beta_i\in H^2(F,\ZZ)$  the assertion follows
from
$$(\alpha_1.\alpha_2)=2(\beta_1.\beta_2)-2(\beta_1.\iota^\ast\beta_2),$$
which uses that $\iota^\ast$ is an isometric involution.
\end{proof}

\subsection{} The linear system of all quintics $|\ko(5)|$ is of dimension $55$. Dividing by the natural action of ${\rm PGL}(4)$ and imposing the existence of $16$ nodes reduces the dimension to $24$. For the details of this dimension count, in a more general setting, see \cite[Prop.\ 2.26]{Catanese}. 

As we will recall below, the family of even $16$-nodal quintics constructed from cubic fourfolds is also  of dimension $24$. In fact, due to results of Catanese \cite[Thm.\ 3.3\, \&\, Rem.\ 3.4]{Catanese}, the space of even $16$-nodal quintics is irreducible. Therefore, the generic even $16$-nodal quintic $D\subset \PP^3$ comes from a Fano variety of lines on a cubic fourfold, by the construction to be recalled in the next section. In particular, any additional topological information
obtained in the Fano setting, holds for any even $16$-nodal quintic and its natural double cover.

\begin{remark} The situation is similar to the case of cubic threefolds. The generic smooth plane quintic curve
is indeed of the form $D_L$. Note however that the generic \'etale double quotient
$C\to D$ of a curve of genus $11$ is not of the form $C_L\to D_L$ and the generic deformation of $C$ alone will not
have any quotient of this form. Similarly, for cubic fourfolds, the generic (infinitesimal) deformation of
$F$ will not be a double cover of a nodal quintic. Indeed, a computation using the Hirzebruch--Riemann--Roch theorem reveals that $\chi(F,\kt_F)=-40$, and, therefore,
$h^1(F,\kt_F)\geq 40$.\footnote{This comment was prompted by a question of G.\ Oberdieck.}
\end{remark}

\section{The Fano perspective}\label{sec:Fanopersp}

In the following, $X\subset\PP^5$ is a smooth cubic fourfold and $F(X)$ denotes its Fano variety of lines. According to Beauville and Donagi \cite{BD}, $F(X)$ is a hyperk\"ahler fourfold and
we will use $\sigma\in H^{2,0}(F(X))$ to denote a non-degenerate holomorphic symplectic form on it.\smallskip

For a  line $L\subset X$ let
 $$F_L\coloneqq\{~L'\mid L\cap L'\ne\emptyset~\}\subset F(X)$$ 
be  the Fano variety of all lines contained in $X$ that intersect $L\subset X$.
  
\subsection{}  From \cite{VoisinGT} we recall the following.
For a generic line $L\subset X$ in a smooth cubic fourfold $X\subset\PP^5$ the Fano variety $F_L$ is a smooth surface of general
type. Mapping $L'\in F_L$ to the intersection point of the plane $\overline{LL_1}\cong\PP^2$
with a fixed generic $\PP^3$ defines a finite morphism
\begin{equation}\label{eqn:FLDL}
\pi\colon F_L\twoheadrightarrow D_L\subset \PP^3
\end{equation} of degree two onto a quintic surface $D_L\subset \PP^3$ with exactly $16$ nodes $x_1,\ldots,x_{16}\in D_L$ as its only singularities. The covering involution $\iota\colon F_L\congpf F_L$ has exactly $16$ fixed points $x_1,\ldots,x_{16}\in F_L$ mapping bijectively onto the nodes of $D_L$. In other words, $D_L$ is an even $16$-nodal quintic and $F_L\to D_L$ is its natural double cover.\smallskip

\begin{proof}
Indeed, the surface $D_L$ is the discriminant divisor of the linear projection
${\rm Bl}_L(X)\twoheadrightarrow \PP^3$ of $X$ from the line $L\subset X$.  By construction,
the covering involution maps a line $L'\in F_L$ to the residue line
of $L\cup L'\subset X\cap \overline{LL'}$, i.e.\ $L\cup L'\cup \iota(L')=X\cap \PP^2$
for a certain plane $\PP^2\subset \PP^5$.
For generic $L$ the singularities of the discriminant surface $D_L$ are nodes
which correspond exactly to the fixed points of $\iota$. The number of these points is
computed by Porteous's formula, see the original \cite{VoisinGT} or \cite[Sec.\ 6.4.5]{HuyCubics} for more details.
\end{proof}

By definition, the line $L$ defines a point in $F_L$.\footnote{This is in contrast to the case of cubic threefolds, where
for a generic line $L$ the set all lines intersecting $L$ splits off the point corresponding to $L$.} For a generic choice, this is not a fixed point of the involution $\iota$.

\subsection{}\label{sec:NotLag} The restriction
$\sigma|_{F_L}\in H^{2,0}(F_L)$ of the holomorphic symplectic form
$\sigma\in H^{2,0}(F(X))$ is not zero, i.e.\ $F_L\subset F(X)$ is not Lagrangian.
\smallskip

\begin{proof} Indeed, due to \cite{VoisinGT} one knows that
$3\,[F_L]$ is the image of $h^3=[\PP^2]\in H^6(X,\ZZ)$ under the Fano correspondence and, therefore, $[F_L]=(1/3)(g^2-[F(Y)])$. Here, $g$ denotes the Pl\"ucker polarization and $F(Y)\subset F(X)$ is the Fano surface of a generic hyperplane section $Y=X\cap \PP^4$, see \cite{Amerik} or \cite[Sec.\ 2.5.1\,\&\,6.4.1]{HuyCubics}. Using the positivity of $g$ and the fact that $F(Y)\subset F(X)$ is Lagrangian \cite{VoisinLagr}, one concludes
$$\int_{F_L}(\sigma\bar\sigma)|_{F_L}=(1/3)\int_{F(X)}(\sigma\bar\sigma)g^2\ne0,$$
which in particular proves the assertion $\sigma|_{F_L}\ne0$.
\end{proof}

\subsection{}\label{sec:primanti}
 The restriction of any primitive class $\alpha\in H^2(F(X),\ZZ)_{\rm pr}$ to $F_L$
is anti-invariant with respect to the action of $\iota$ on $H^2(F_L,\ZZ)$, i.e.\ $\iota^*(\alpha|_{F_L})=-\alpha|_{F_L}$. 
\smallskip

\begin{proof}
The proof imitates the well-known argument for the analogous fact for cubic threefolds, cf.\  \cite[Sec.\ 5.3.1]{HuyCubics} for an account and references. First note that the assertion is invariant
under deformations. Next observe that the union of all $H^{2,0}(F(X'))\subset H^2(F(X),\CC)_{\rm pr}$, for arbitrary smooth deformations $X'\subset\PP^5$  of $X$, is Zariski dense.
Thus, it suffices to
show that $\iota^*(\sigma|_{F_L})=-\sigma|_{F_L}$ or, dually, that the composition
$$H^{2,0}(D_L)\to H^{2,0}(F_L)\to H^{4,2}(F(X))\to H^{3,1}(X)$$ is zero. This follows
from the Bloch--Srinivas principle, see for example \cite[Prop.\ 22.24]{VoisinHodge}, and the observation that the map
$D_L\to \CH_1(X)$ is constant. To see the latter, consider $L'\in F_L$ and  let
$t\coloneqq \pi(L')\in D_L$.
Then, $L\cup L'\cup \iota(L')=X\cap \PP^2$ 
and the image of the point $t\in D_L$ is the constant class $[X\cap \PP^2]-[L]\in \CH_1(X)$.
\end{proof}

\subsection{}\label{sec:crkone}
Restriction defines an injection of Hodge structures of corank one
$$H^2(F(X),\ZZ)_{\rm pr}\,\hookrightarrow H^2(F_L,\ZZ)^-.$$

\begin{proof}
By \ref{sec:primanti}, the restriction of any primitive class is indeed anti-invariant.
To prove the injectivity of the restriction map one can use a standard
deformation argument, cf.\ \cite[Rem.\ 2.5.7]{HuyCubics}.
As the assertion is purely topological, we may assume
that $X$  is very general and, in particular, that $H^2(F(X),\ZZ)_{\rm pr}$ is an
irreducible Hodge structure. Hence, the restriction map
is either trivial or injective. However, as $F_L\subset F(X)$ is not Lagrangian by \ref{sec:NotLag},
i.e.\ $\sigma|_{F_L}\ne0$, and $\sigma\in H^{2,0}(F(X))\subset
H^2(F(X),\ZZ)_{\rm pr}\otimes\CC$, it is certainly not trivial. Alternatively, one can use
\ref{sec:isom} below.

The assertion on the corank follows from $\rk\, H^2(F(X),\ZZ)=23$ and \ref{sec:invpart}.
\end{proof}

\subsection{}\label{sec:isom}
For any class $\alpha\in H^2(F(X),\ZZ)_{\rm pr}$ and its restriction
$\alpha|_{F_L}\in H^2(F_L,\ZZ)$ one has
$$2\,q(\alpha)=(\alpha|_{F_L}.\alpha|_{F_L}).$$
Here, $q$ is the Beauville--Bogomolov--Fujiki form on the hyperk\"ahler fourfold $F(X)$ and $(~.~)$ denotes the
intersection form on the surface $F_L$.
\smallskip

\begin{proof}
The result follows from a straightforward and well-known
computation of certain natural cohomology classes on $F(X)$,
see \cite[Sec.\ 6.4.1]{HuyCubics} for an account and references,
and a result of Voisin mentioned before: The natural class $\tilde q\in H^4(F(X),\QQ)$
defined by the condition that $q(\alpha)=\int_{F(X)}\alpha^2\cdot\tilde q$
satisfies $$30\,\tilde q={\rm c}_2(\kt_{F(X)})=15\,[F_L]-3\,{\rm c}_2(\ks_F)=15\,[F_L]-3\,[F(Y)],$$
where as above $F(Y)\subset F(X)$ is the  surface of all lines
contained in a generic hyperplane section $Y=X\cap H$.
Hence,  $q(\alpha)=\int_{F(X)}\tilde q\cdot \alpha^2=(1/2)\int_{F_L}\alpha|_{F_L}^2$. Here, the last equality follows from the vanishing $\alpha|_{F(Y)}=0$, and hence $\int_{F(Y)}\alpha|_{F(Y)}^2=0$,
 for all primitive classes $\alpha$, which is a consequence of $F(Y)\subset F(X)$ being Lagrangian \cite[Ex.\ 3.7]{VoisinLagr}.
 \end{proof}

\subsection{}\label{sec:gpm}  Let $g={\rm c}_1(\ko(1))\in H^2(F(X),\ZZ)$ be the class of the Pl\"ucker polarization and let $g|_{F_L}=g^++g^-$ be the decomposition of its restriction in its invariant and anti-invariant part, i.e.\ $g^\pm\in H^2(F_L,\QQ)$ with $\iota^*g^\pm=\pm g^\pm$. Then both parts are non-zero, i.e.\ $g^\pm\ne0$. 
 \smallskip
 
\begin{proof} Since $g|_{F_L}$ is ample, so is $\iota^*g|_{F_L}$. Hence,
$\iota^*g|_{F_L}\ne - g|_{F_L}$, i.e.\ $g^+\ne0$. Suppose now that $g^-=0$. Then $g$ is $\iota$-invariant. Since $\Pic(F_L)\subset H^2(F_L,\ZZ)$ torsion free, see Lemma \ref{lem:tfH2}, it means that the restriction $\ko(1)$
of the Pl\"ucker polarization  to $F_L$ is $\iota$-invariant.
Hence,  its restriction to $F_L'\coloneqq F_L\setminus\{x_i\}$ descends to a line bundle on $D_L'=D_L\setminus\{x_i\}$, cf.\
 \cite[Lem.\ 0.3.4]{BeauvPrym} or \cite[Lem.\ 3.1]{OP}. Combining this with \ref{sec:PicDD} and the fact that $\pi^*(\kl|_{D_L'})\cong\ko_{F_L'}$,  one finds a line bundle $\kk$ on $D_L$ such that $\pi^\ast(\kk|_{D_L'})\cong \ko(1)|_{F'_L}$ and, in fact, $\pi^\ast\kk\cong\ko(1)|_{F_L}$, as $\Pic(F_L)\cong \Pic(F_L')$. Thus, for $k\coloneqq{\rm c}_1(\kk)$, we have
  $\pi^\ast k= g|_{F_L}$ and, in particular,
 $(\pi^\ast k.\pi^\ast k)= (g|_{F_L}.g|_{F_L})$, which we will show to be impossible.
 
 First note that $(\pi^\ast k.\pi^\ast k)=2\cdot (k.k)$ is even, as $F_L\twoheadrightarrow D_L$ is of degree two.
  On the other hand, 
 $3\, [F_L]$ is the image of $h^3\in H^6(X,\ZZ)$ under the Fano correspondence
  and $[F_L]=(1/3)\,(g^2-[F(Y)])$, cf.\ \cite[Sec.\ 2.5.1\,\&\,6.4.1]{HuyCubics}. This implies $$\resizebox{0.95\hsize}{!}{$(g|_{F_L}.g|_{F_L})=\int_{F(X)}[F_L]\cdot g^2=\frac{1}{3}\left(\int_{F(X)}g^4-\int_{F(Y)}g^2|_{F(Y)}\right)=\frac{1}{3}(\deg(F(X))-\deg(F(Y)))=21$}.$$ Thus, $(g|_{F_L}.g|_{F_L})$ is  odd, which
  produces the desired contradiction.
  \end{proof}
  
\begin{remark} To show that $g$ is not invariant, one can alternatively use that the linear system $\ko(1)|_{F_L}\otimes\pi^\ast\ko(-1)$ defines the rational map $\xymatrix@C=18pt{F_L\ar@{..>}[r]& L}$ that maps $L'$ to its point of intersection with $L$, which is clearly not $\iota$-invariant. See \cite[\S3, Lem.\ 2]{VoisinGT}.
\end{remark}

 \begin{remark}
 (i) The arguments in the proof also show that every $\iota$-invariant line bundle on $F_L$
 is the pull-back of a line bundle on $D_L$. A priori, this is not clear, as a linearization
 of an invariant line bundle
 may act non-trivially on the fibre at one of the fixed points. 
 \smallskip
 
 (ii) Also, it is not
 clear whether the same holds true for arbitrary cohomology classes. In other words,
 is $H^2(D_L,\ZZ)\to H^2(F_L,\ZZ)^+$  surjective? One way to prove surjectivity would be to show that $H^2(F_L,\ZZ)^+$  is generated by classes that become algebraic on some deformation of $D_L$ as an even $16$-nodal quintic. Note that in the analogous
 situation for cubic threefolds the map is indeed not surjective. Indeed, for
 an \'etale  double cover $C\twoheadrightarrow\bar C$ of smooth curves
 the map $H^1(\bar C,\ZZ)\to H^1(C,\ZZ)^+$ has a cokernel of order two, see \cite[Sec.\ 5.3.2]{HuyCubics} for references.
 \end{remark}

 Consider the torsion free lattice  $$H^2(F_L,\ZZ)^-_{\rm pr}\subset H^2(F_L,\ZZ)$$
of all classes that  are primitive with respect to the restriction $g|_{F_L}$ of the Pl\"ucker polarization and anti-invariant with
 respect to the involution $\iota$. 
 
 \begin{cor}
 The lattice $H^2(F_L,\ZZ)^-_{\rm pr}$ is of rank $22$ and signature $(2,20)$. It is
 naturally endowed with a Hodge structure of K3 type.
 \end{cor}
 
 \begin{proof} As $\rk\, H^2(F_L,\ZZ)^-=23$ by  \ref{sec:invpart}, it suffices to show that $H^2(F_L,\ZZ)^-_{\rm pr}\subset H^2(F_L,\ZZ)^-$ is a proper
 sub-lattice, i.e.\ the linear form $(g.~)$ on  $H^2(F_L,\ZZ)^-$ is non-zero. As the two sublattices $H^2(F_L,\ZZ)^-$ and $H^2(F_L,\ZZ)^+$ are orthogonal with respect to $(~.~)$, this follows from $g^-\ne0$  proved above. By the Hodge index theorem, $(g^-.g^-)<0$, which implies the claimed signature. That the Hodge structure is of K3 type follows from \ref{sec:decompH2}.
\end{proof}

\section{Hodge isometries and global Torelli}\label{sec:proof}
The goal of this section is to prove Theorem \ref{thm:FLF}.

\subsection{} As before, we consider a smooth cubic fourfold $X\subset\PP^5$ and a generic line $L\subset X$.
\begin{prop}\label{prop:MainHodge}
The restriction map induces an isometry of Hodge structures of K3 type
$$(H^2(F(X),\ZZ)_{\rm pr},q)\cong (H^2(F_L,\ZZ)_{\rm pr}^-,(1/2)(~.~)).$$
\end{prop}

\begin{proof} 
By virtue of  \ref{sec:crkone} and \ref{sec:isom}, restriction embeds
the Hodge structure $H^2(F(X),\ZZ)_{\rm pr}$ endowed with the Beauville--Bogomolov--Fujiki form $q$
isometrically into $H^2(F_L,\ZZ)^-$ viewed with the scaled intersection form $(1/2)(~.~)$. The latter is
integral on $H^2(F_L,\ZZ)^-$ due to Corollary \ref{cor:evenInt}.

This leads to a morphism of rational Hodge structures
$$\xymatrix{H^2(F(X),\QQ)_{\rm pr}\ar@{^(->}[r]&H^2(F_L,\QQ)^-=H^2(F_L,\QQ)^-_{\rm pr}\oplus\QQ\cdot g^-\ar@{->>}[r]&\QQ\cdot g^-,}$$
which  has to be trivial due to the irreducibility of the Hodge structure $H^2(F(X),\QQ)_{\rm pr}$ 
for the very general $X$. Hence, restriction defines an embedding of Hodge structures 
\begin{equation}\label{eqn:HodgeIsoso}
H^2(F(X),\ZZ)_{\rm pr}\,\hookrightarrow H^2(F_L,\ZZ)^-_{\rm pr},
\end{equation}
which is isometric with respect to the two symmetric forms $q$ and $(1/2)(~.~)$. Note that both
sides are of rank $22$. Thus, restriction identifies   $H^2(F(X),\ZZ)_{\rm pr}$
with a  sub-lattice  of  $H^2(F_L,\ZZ)_{\rm pr}^-$ of  finite index $m<\infty$.
By a standard fact in lattice theory, the discriminants of the two lattices are related
by the formula  ${\rm discr}(q)={\rm discr}((1/2)(~.~))\cdot m^2$. 

Since 
by \cite{BD} there exists an isometry (up to a global sign) $H^2(F(X),\ZZ)_{\rm pr}\cong H^4(X,\ZZ)_{\rm pr}$, we know that ${\rm discr}(q)=3$. This suffices to conclude  that $m=1$, i.e.\ (\ref{eqn:HodgeIsoso}) is bijective
\end{proof}

Together with the result of Beauville and Donagi \cite{BD} this proves Theorem  \ref{thm:FLF}.\qed

\begin{remark}
(i) The isomorphism types of the two lattices $H^2(F(X),\ZZ)$ and $H^2(F(X),\ZZ)_{\rm pr}$ are known. The latter is the even lattice $E_8(-1)^{\oplus 2}\oplus U^{\oplus 2}\oplus A_2(-1)$, which therefore also describes $H^2(F_L,\ZZ)^-_{\rm pr}$. 

However, it is not clear to us how to determine the isomorphism type of the lattice $H^2(F_L,\ZZ)^-$, which according to \ref{sec:invpart} has signature $(2,21)$. It is tempting to conjecture the existence of an isometry
(up to a global sign) $H^4(X,\ZZ)\cong H^2(F_L,\ZZ)^-$, but I have no further evidence for it.\smallskip

(ii) A.\ Kuznetsov suggested that Proposition  \ref{prop:MainHodge} could possibly be seen as a consequence of a
result for general conic fibrations $\phi\colon\tilde X\to\PP^3$. Indeed, it seems feasible that
one can establish a direct link between the Hodge structures  $H^4(\tilde X,\ZZ)$ and $H^2(F,\ZZ)^-$,
where $F$ is the natural double cover of the discriminant surface $D\subset\PP^3$ of $\phi$. 
Due to the dependence on the Pl\"ucker polarization, incorporating the quadratic forms and defining the primitive part of $H^2(F_L,\ZZ)^-$ seem less obvious. In any case, in  the context of a cubic fourfold $X$ the approach via its
Fano variety $F(X)$ has the advantage  
that the Hodge isometry constructed above is compatible with the inclusion $F=F_L\subset F(X)$.

 %
 %
\end{remark}

\begin{remark}\label{rem:deferal}
General deformation theory for hyperk\"ahler manifolds implies that the classes in $H^2(F(X),\ZZ)_{\rm pr}$ that are of type $(1,1)$ for some deformation of $X$
generate the full primitive cohomology. Hence, the same is true for $H^2(F_L,\ZZ)^-$ and then
for $H^2(F,\ZZ)^-$ of the natural cover of an arbitrary even $16$-nodal quintic. This was used in the proof of Lemma \ref{lem:antiinvH}.
\end{remark}

\begin{remark}
According to Hassett \cite{Hassett}, there exists a countable union of divisors in the moduli space of cubic fourfolds
for which $H^2(F(X),\ZZ)_{\rm pr}$ contains the Hodge structure $H^2(S,\ZZ)_{\rm pr}$ of a polarized K3
surface. The correspondence is known to be algebraic \cite{AT}. Thus, along these divisors one finds algebraic
correspondences between a K3 surface $S$ and the surface of general type $F_L$. Can those
be realized geometrically?
\end{remark}

\subsection{}\label{sec:spellout}
It may be worthwhile to spell out the various Hodge theoretic conditions and their geometric consequences. Let  $L\subset X$ and  $L'\subset X'$ be as above and let us consider 
the following statements:


\begin{enumerate}
\item[{\rm (i)}] There exists an isomorphism $F_L\cong F_{L'}$ compatible with the natural involutions $\iota$ and $\iota'$.\vskip0.3cm
\item[{\rm (ii)}] There exists an equivariant isometry  
$(H^2(F_L,\ZZ)_{\rm pr},(~.~),\iota^\ast)\cong (H^2(F_{L'},\ZZ)_{\rm pr},(~.~),\iota'^\ast)$ of Hodge structures.\vskip0.3cm
\item[{\rm (iii)}] There exists an isometry of Hodge structures 
$(H^2(F_L,\ZZ)^-_{\rm pr},(~.~))\cong (H^2(F_{L'},\ZZ)^-_{\rm pr},(~.~))$.\vskip0.3cm
\item[{\rm (iv)}] There exists an isomorphism of polarized varieties
$(F(X),g)\cong (F(X'),g')$.\vskip0.3cm
\item[{\rm (v)}]  There exists an isomorphism $X\cong X'$.
\end{enumerate}
\smallskip

Then the following implications hold
$$ \text{ (i) } \Rightarrow\text{ (ii) } \Rightarrow\text{ (iii) }\Leftrightarrow\text{  (iv) } \Leftrightarrow\text{ (v).}$$

The equivalence of (iii) and (iv) follows from Theorem \ref{thm:FLF} and the global Torelli theorem
for hyperk\"ahler manifolds. All other implications are either obvious or well known. The first implication is not an equivalence and neither should be the second.

Corollary \ref{cor:GT1} is the combination (iii) $\Leftrightarrow$ (iv) $\Leftrightarrow$ (v). It is the analogue of the celebrated result of Clemens--Griffiths and Tyurin combined with Mumford's work on Prym varieties: For two smooth cubic threefolds $Y$ and $Y'$, one has
$$Y\cong Y'~\Leftrightarrow ~ {\rm Prym}(C_L/D_L)\cong{\rm Prym}(C_{L'}/D_{L'})$$
with an isomorphism of polarized abelian varieties on the right hand side. Here, $L\subset Y$ and $L'\subset Y'$ are generic lines and $C_L\to D_L$ and $C_{L'}\to D_{L'}$ are the analogues of $F_L\to D_L$ and $F_{L'}\to D_{L'}$, see \cite[Ch.\ 5]{HuyCubics} for references.

\subsection{}\label{sec:IzadiShen}
Theorem \ref{thm:FLF} can be seen as a stronger and more precise version
(in dimension four) of  a result of Izadi \cite[Thm.\ 3]{Izadi}, which asserts the existence of an exact
sequence $$0\to H^2(D_L,\ZZ)_{\rm pr}\to H^2(F_L,\ZZ)_{\rm pr}\to H^4(X,\ZZ)_{\rm pr}\to 0,$$
where the ample class $g|_{F_L}+{\rm c}_1(\pi^\ast\ko(1))$  is used to define the primitive part.
However, the proof of the surjectivity is incomplete. It uses the surjectivity of $H^2(F(X),\ZZ)\twoheadrightarrow H^2(F(X),\ZZ)_{\rm pr}$ claimed in \cite[Sec.\ 4]{Izadi}, which only exists with coefficients in $\QQ$ or after some suitable localization, and, no argument is given for the surjectivity in the proof of \cite[Lem.\ 5.12]{Izadi}.
The problem is similar to showing that $H^2(F(X),\ZZ)_{\rm pr}\,\hookrightarrow H^2(F_L,\ZZ)_{\rm pr}^-$ is bijective, see the proof of Proposition
\ref{prop:MainHodge}. Furthermore, the claim that $H^2(D_L,\ZZ)$ is the invariant part of $H^2(F_L,\ZZ)$, which again does hold for coefficients in $\QQ$ and for algebraic classes, see \ref{sec:invpart}, \ref{sec:gpm}, is not adequately addressed and no proof is given for the torsion freeness of $H^2(F_L,\ZZ)$.
\smallskip

 Shen's result \cite[Thm.\ 1.5 (2)\,\&\,Rem.\ 5.10]{Shen2} for the Fano surface of lines meeting a fixed general rational curve of degree at least two is similar to the above proposition. The proof of the surjectivity there relies  on degeneration techniques developed in \cite{Shimada}. Also the image of $1-\iota^\ast$ is used instead of $H^2(F_L,\ZZ)_{\rm pr}^-$, which, however, by virtue
of \ref{sec:antinv} eventually amounts to the same. Subsequently, Shen considered the case
of lines. The results \cite[Thm.\ 4.7 \& Cor.\ 4.8]{Shen0} come closest to Proposition
\ref{prop:MainHodge} and Theorem \ref{thm:FLF},
although Shen defines the primitive anti-invariant part in terms of two classes in $H^2(F_L,\ZZ)$, the Pl\"ucker polarization and the class of the fibre of the natural projection $\xymatrix@C=18pt{F_L\ar@{..>}[r]&L}$. Also, the convention for the pairings  $q$ and $(1/2)(~.~)$ on the two sides are different. In \ref{sec:ShenBP} we give a brief comparison of  the techniques.

\section{Chow groups and Chow motives}\label{sec:Chow}
The goal of this section is to `lift' the Hodge isometry constructed above to the level of integral Chow groups and rational Chow motives. As a first step, one needs to define properly the analogue of the Prym variety ${\rm Prym}(C_L/D_L)$ as a subgroup of the Chow group $\CH_0(F_L)$.

\subsection{}\label{sec:FDgen}
Let us first consider the general situation of an even $16$-nodal quintic $D\subset\PP^3$
and the natural double cover $\pi\colon F\twoheadrightarrow D$ with the covering involution
$\iota$.

Let $\CH(F)_{\rm hom}=\CH_0(F)_{\rm hom}\subset\CH_0(F)$ be  the homologically trivial part of the Chow group (of $0$-cycles) on the surface $F$. We consider it with the action induced by the involution $$\iota^\ast\colon \CH_0(F)_{\rm hom}\congpf
 \CH_0(F)_{\rm hom}$$ 
 and define the two subgroups $$\CH_0(F)_{\rm hom}^\pm\coloneqq\{~\alpha\mid\iota^*\alpha=\pm\alpha ~\}.$$

The groups  $\CH_0(D)_{\rm hom}$ and $\CH_0(F)_{\rm hom}$ are divisible and, 
since the Albanese of the regular surfaces $D$ and $F$ are trivial, 
also torsion free \cite{Bl,Roi}. In particular, this allows one to write any class $\alpha\in \CH_0(F)_{\rm hom}$ as $$\alpha=\alpha^++\alpha^-$$
with integral $\alpha^\pm\in \CH_0(F)_{\rm hom}^\pm$. Explicitly, $\alpha^\pm\coloneqq (1/2)(\alpha\pm\iota^*\alpha)$. In other words,
\begin{equation}\label{eqn:CHpm}
\CH_0(F)_{\rm hom}=\CH_0(F)_{\rm hom}^+\oplus \CH_0(F)_{\rm hom}^-.
\end{equation}

The group $\CH_0(F)_{\rm hom}^+$ can be identified with $\CH_0(D)_{\rm hom}$
via $$\pi^\ast\colon \CH_0(D)_{\rm hom}\congpf\CH_0(F)_{\rm hom}^+.$$
Indeed, $\pi^\ast$ is injective, as $\pi_\ast\circ\pi^\ast=2\cdot{\rm id}$ and $\CH_0(D)_{\rm hom}$
is torsion free. The surjectivity follows from the well-known statement for Chow groups
with coefficients in $\QQ$, see \cite[Exa.\ 1.7.6]{Fulton}, and the divisibility of the Chow group.
Note that the singularities of the surface $D$ do not cause trouble, for
$\CH_0(D)_{\rm hom}\cong \CH_0(\tilde D)_{\rm hom}$ by \cite[Cor.\ 9.8]{Srini}. 
\smallskip

We are more interested in the anti-invariant part which admits several alternative descriptions.

\begin{lem}\label{lem:ChF-}
Projection  defines an isomorphism
$$\CH_0(F)_{\rm hom}^-\cong \CH_0(F)_{\rm hom}/ \CH_0(F)_{\rm hom}^+\cong 
\CH_0(F)_{\rm hom}/ \CH_0(D)_{\rm hom}.$$
Furthermore,
$$\CH_0(F)_{\rm hom}^-=\im\left(1-\iota^\ast\colon  \CH_0(F)_{\rm hom}\to  \CH_0(F)_{\rm hom}\right)$$
and 
$$\CH_0(F)_{\rm hom}^-=\ker\left(\pi_\ast\colon\CH_0(F)_{\rm hom}\to \CH_0(D)_{\rm hom}\right).$$
In particular, $\CH_0(F)_{\rm hom}^-$ is generated by classes of the form
$[t]-[\iota(t)]\in \CH_0(F)$.

\end{lem}
\begin{proof} The first two assertions follow from the discussion above. It remains to verify
the description  of $\CH_0(F)_{\rm hom}^-$ as $\ker(\pi_\ast)$.

As $\pi_\ast\alpha=\pi_\ast\iota^\ast\alpha=\pi_\ast(-\alpha)$ for all $\alpha\in
 \CH_0(F)_{\rm hom}^- $ and since $\CH_0(D)_{\rm hom}$ is torsion free, the inclusion
$\CH_0(F)_{\rm hom}^-\subset\ker(\pi_\ast)$ is clear. Conversely, write a given $\alpha\in\ker(\pi_\ast)$ 
as above as $\alpha=\alpha^++\alpha^-=\pi^\ast\beta+\alpha^-$, which leads to $0=\pi_\ast\alpha=\pi_\ast\pi^\ast\beta=2\beta$ and, therefore,
$\beta=0$, i.e.\ $\alpha=\alpha^-$.
\end{proof}

\begin{remark}
The situation is very similar to the case of cubic threefolds with
$\CH_0(F_L)_{\rm hom}^-$ replacing ${\rm Prym}(C_L/D_L)$. However,
there are two notable differences caused by the torsion freeness of the Chow group. 

First,
the Prym variety ${\rm Prym}(C_L/D_L)$, defined as the image of
$1-\iota^\ast\colon \Pic^0(C_L)\to \Pic^0(C_L)$ and thus isomorphic to the quotient
$\Pic(C_L)/\Pic(C_L)^+$,  is only one of the two connected components of
the kernel of $\pi_\ast\colon\Pic^0(C_L)\to\Pic^0(D_L)$. Second,
$\pi^\ast\colon \Pic^0(D_L)\to\Pic^0(C_L)$ has a kernel of order two generated by the torsion line bundle defining the \'etale double cover $C_L\to D_L$.
In other words, the analogue of (\ref{eqn:CHpm}) in the case of cubic threefolds is not
a direct product decomposition of $\Pic^0(C_L)$ but the \'etale degree two
map $\Pic^0(D_L)\times {\rm Prym}(C_L/D_L)\twoheadrightarrow\Pic^0(C_L)$.
\end{remark}

\subsection{} The arguments to prove the first part of Theorem \ref{thm:mainChow}
follow closely ideas of Murre \cite{Murre}. At the heart of the proof is the following
geometric construction which has immediate consequences for the Chow group
of $X$. Here, as before, we use  $L,L'$, etc.\ to denote a line in $X$ as well as the corresponding point in $F_L$.

Consider a tangent direction $v$ at a point $x\in L\subset X$,
i.e.\ a line in $T_xX$, and let $L_v\subset\PP^5$ be the unique line through $x$ realizing this tangent direction. If $L_v$ is not contained in $X$, then it intersects $X$ in $x$ with multiplicity at least two and, therefore, defines a unique point $y_v\in X$ such that
$L_v\cap X=2x+y_v$, cf.\ \cite[Sec.\ 2.1.5]{HuyCubics}. This defines a dominant rational map $\xymatrix{\PP(\kt_X|_L)\ar@{..>}[r]&X}$ which is not defined at the points $v$ with $L_v\subset X$. Note that
by construction any $L_v\subset X$, $v\in \PP(\kt_X|_L)$, intersects $L$ and hence defines a point $L_v\in F_L$. Conversely, mapping a line $L'\in F_L$ distinct from $L$ to its tangent
direction at the point of intersection $L\cap L'=\{x\}$ defines a map
$F_L\setminus\{L\}\,\hookrightarrow \PP(\kt_X|_L)$, $L'\mapsto (x,T_xL')$ which extends
to a closed embedding $${\rm Bl}_{L}(F_L)\,\hookrightarrow \PP(\kt_X|_L)$$
of the blow-up ${\rm Bl}_{L}(F_L)$ of $F_L$ in the point $L\in F_L$. The blow-up
can also be described as $q^{-1}(L)$, where $\xymatrix{F&\ar[l]_-p \LL
\ar[r]^-q& X}$ denotes the universal family of lines in $X$. In particular, the exceptional 
curve in ${\rm Bl}_{L}(F_L)$ can be viewed as $p^{-1}L\subset\LL$. The picture
looks as follows:
$$\hskip-2cm\xymatrix@C=18pt@R=15pt@M=8pt{\hskip2cm\LL\ar@{^(->}[r]&\PP(\kt_X)\\
{\rm Bl}_{L}(F_L)\cong q^{-1}(L)\ar@{^(->}[r]
\ar@{^(->}[]!<5.9ex,0ex>;[u]!<5.9ex,0ex>
&\PP(\kt_X|_L)\ar@{^(->}[u]\\
\hskip1.5cm L\cong p^{-1}L\ar@{^(->}[]!<5.9ex,0ex>;[u]!<5.9ex,0ex>\ar[r]^-\sim
&
\PP(\kt_L).\ar@{^(->}[u]
}$$
We now consider the blow-up
$$\tau\colon \tilde X\coloneqq{\rm Bl}_{{\rm Bl}_L(F_L)}(\PP(\kt_X|_L))\twoheadrightarrow \PP(\kt_X|_L)$$ with
the surface ${\rm Bl}_{L}(F_L)$ as its center. Then the dominant rational map $\xymatrix{\PP(\kt_X|_L)\ar@{..>}[r]&X}$ extends to a surjective morphism $$\gamma\colon \tilde X\twoheadrightarrow X.$$
Alternatively, $\tilde X$ can be described as the incidence variety of all triples
$(x,y,L')$ consisting of a line $L'\subset\PP^5$ and points $x\in L\cap L'$,
$y\in X\cap L'$ such that $2x\subset X\cap L'$, i.e.\ $L'$ is tangent to
$X$ at the point $x$.\footnote{A.\ Kuznetsov suggested to view $\tilde X$  alternatively
as the fibre product ${\rm Bl}_L(X)\times_{\PP^3}E$, where $E\subset{\rm Bl}_L(X)$ denotes the exceptional divisor.
Note that the restriction $\phi\colon E\to \PP^3$ is generically finite of degree two and that the conic fibration
given by the first projection
$\tilde X\to E$ comes with a section.} Then $$\tau(x,y,L')=(x,T_xL')\in \PP(\kt_X|_L)\text{ and }\gamma(x,y,L')=y\in X.$$ For $L'\in F_L\setminus\{L\}$ thought of
as the point $(x,T_xL')\in\PP(\kt_X|_L)$, the fibre $\tau^{-1}(L')\subset \tilde X$ consists all triples $(x,y,L')$, where $\{x\}=L\cap L'$ and $y\in L'$ arbitrary.
In particular, $\gamma\colon\tau^{-1}(L'=(x,T_xL'))\congpf L'\subset X$.
Similarly, for $(x,T_xL)\in \PP(\kt_X|_L)$ the fibre under the blow-up
is $\{(x,y,L)\mid y\in L\}$. Hence, the pre-image $\tau^{-1}(p^{-1}L)$ of the exceptional line $L\cong p^{-1}L\subset {\rm Bl}_L(F_L)$ is naturally identified
with $L\times L$ such that $\tau$ and $\gamma$ correspond to the two-projections.
\smallskip

The morphism  $\gamma$ is generically of degree two, which
is seen as follows. For any point $y\in X\setminus L$ consider the plane cubic curve
$X\cap\overline{Ly}$. For a generic point $y$ the residual conic $Q_y$ of $L\subset X\cap\overline{Ly}$ does not contain $L$ and, therefore, intersects $L$ in at most two points $x_1,x_2\in L$. The tangent directions of the two lines connecting 
$x_1$ and $x_2$ with $y$ are the pre-images of $y$. See \cite[Cor.\ 2.1.21]{HuyCubics} 
for further details.

The induced rational covering involution
$$\xymatrix@R=10pt{\tilde X\ar[dr]_-\gamma\ar@{..>}[rr]^-j&&\tilde X\ar[dl]^-\gamma\\
&X&}$$ maps a generic point $(x_1,y,L)$ to $(x_2,y,L')$,
where $L'$ is the line through $y$ and the second point of intersection of $Q_y\cap L=\{x_1,x_2\}$.

\subsection{} Since $\tau\colon \tilde X\to\PP(\kt_X|_L)$ is the blow-up in ${\rm Bl}_L(F_L)\subset \PP(\kt_X|_L)$, the Chow group $\CH_1(\tilde X)$
naturally splits as
$$\begin{array}{rclcl}
\CH_1(\tilde X)&\cong &\tau^\ast \CH_1( \PP(\kt_X|_L))&\oplus&k_\ast\tau_0^\ast\CH_0({\rm Bl}_L(F_L))\\[6pt]
&\cong&\CH_1( \PP(\kt_X|_L))&\oplus&\CH_0(F_L).
\end{array}$$
Here, $$\tau_0\colon E\coloneqq\tau^{-1}({\rm Bl}_L(F_L))\twoheadrightarrow{\rm Bl}_L(F_L)~\text{ and }~
k\colon\tau^{-1}({\rm Bl}_L(F_L))\,\hookrightarrow \tilde X$$
are the projection, a $\PP^1$-bundle, and the natural closed embedding of the exceptional divisor.

\begin{lem}
The homologically trivial part of $\CH(\tilde X)$ sits in dimension one
and is  naturally identified with the homologically trivial part
of the center of the blow-up $\tau$, i.e.\
$$\CH(\tilde X)_{\rm hom}=\CH_1(\tilde X)_{\rm hom}\cong \CH_0(F_L)_{\rm hom}=\CH(F_L)_{\rm hom}.$$
\end{lem}

\begin{proof}
The assertion follows directly from the fact that the first summand in the above
decomposition injects into cohomology.
\end{proof}

The component in $\CH_0(F_L)$ of a curve class $\alpha\in \CH_1(\tilde X)$   is computed
by $-\tau_{0\ast}\alpha|_E$. For example, if $C\subset \tilde X$ intersects only one fibre
$\tau_0^{-1}(L')$ and does so transversally  in only one point, then $[C]$ is mapped (up to the sign) to
the class of that fibre
$-[\tau_0^{-1}(L')]\in k_*\tau_0^*\CH_0(F_L)$.

\begin{lem}
The direct image $\gamma_\ast\colon \CH(\tilde X)\to \CH(X)$ defines a
surjective map
$$\CH_1(\tilde X)_{\rm hom}\twoheadrightarrow \CH_1(X)_{\rm hom}.$$
Furthermore, $\CH_1(X)$ is  generated by lines intersecting a fixed line $L$.
\end{lem}

\begin{proof}
As $\gamma$ is of degree two, we know that $\gamma_\ast\circ\gamma^\ast=2\cdot{\rm id}$. As $\CH_1(X)$ is generated by lines, see \cite{Paranjape,Shen}
or the more general result  \cite[Thm.\ 1.6]{TianZong}, homological equivalence equals algebraic equivalence. Thus, $\CH_1(X)_{\rm hom}$ is
divisible and the assertion follows. 

The second assertion follows from
$\CH_1(\tilde X)_{\rm hom}$ being contained in the subgroup generated by classes of
 fibres $L'\cong\tau_0^{-1}(L')\subset \tilde X$ over $L'\in {\rm Bl}_L(F_L)$.
\end{proof}

\begin{lem}
The pull-back $\gamma^\ast\colon \CH_1(X)\to \CH_1(\tilde X)$ maps the subgroup
$\CH_1(X)_{\rm hom}\subset \CH_1(X)$ onto $\CH_0(F_L)_{\rm hom}^-\subset\CH_0(F_L)_{\rm hom}\cong
\CH_1(\tilde X)_{\rm hom}$.
\end{lem}

\begin{proof}
We have to show that the image of $\gamma^\ast\colon \CH_1(X)_{\rm hom}\to\CH_0(F_L)$ is 
$\CH_0(F_L)^-$. Clearly, the image is invariant under the covering involution $j^\ast$
and we also know that $\CH_1(X)$ is generated by the classes of lines $L'\in F_L$.

Fix a generic line $L'\in F_L$.   As explained above,
$\tau^{-1}(L')\subset \tilde X$ consists of all triples $(x',y,L')$ with $\{x'\}=L\cap L'$ and $y\in L'$ arbitrary.
By definition, $j(x',y,L')=(x'',y,L'')$, where $x''$ is the point of intersection of $L$ with the residual line $\iota(L')$ of $L\cup L'\subset X\cap\overline{LL'}$, i.e.\ $L\cup L'\cup \iota(L')=X\cap \PP^2$ and $L''\coloneqq\overline{x''y}$.

The second component of the class $[j(\tau^{-1}(L'))]\in \CH_1(\tilde X)=
\CH_1(\PP(\kt_X|_L))\oplus\CH_0(F_L)$ is  $-[\tau^{-1}(L'')]$, because
 $j(\tau^{-1}(L'))$ intersects the exceptional divisor
$\tau^{-1}({\rm Bl}_L(F_L))$ transversally in the point $(x'',z,\iota(L'))$, where
$\{z\}=L'\cap \iota(L')$.
Therefore, the homologically trivial part of $\gamma^\ast[L']$
is nothing but $[\tau^{-1}(L')]+[j(\tau^{-1}(L'))]=[\tau^{-1}(L')]-[\tau^{-1}(\iota(L'))]$,
which corresponds to $[L']-[\iota(L')]\in \CH_0(F_L)^-$. As classes of this form generate $\CH_0(F_L)^-$, this proves surjectivity.
\end{proof}

\begin{lem}\label{lem:2alpha}
Consider a class $\alpha\in \CH_0(F_L)^-\subset \CH_0(F_L)\subset\CH_1(\tilde X)$. Then $\gamma^\ast\gamma_*(\alpha)=2\alpha$.
\end{lem}

\begin{proof}
By the surjectivity of $\gamma^\ast\colon \CH_1(X)_{\rm hom}\twoheadrightarrow \CH_0(F_L)^-$ proved in the previous lemma, it suffices to verify the assertion for classes of the form $\gamma^\ast\beta$, which follows from $\gamma_\ast\circ\gamma^\ast=2\cdot{\rm id}$.
\end{proof}

\subsection{}\label{sec:proofCH}
We can now conclude the proof of the first part of Theorem \ref{thm:mainChow},
by adapting arguments of Murre \cite{Murre} to the case of cubic fourfolds.

First observe that
for $L\ne L'\in F_L$ the morphism $\gamma\colon \tilde X\to X$ induces
an isomorphism $\gamma\colon \tau^{-1}(L')\congpf L'\subset X$.
Hence, for two arbitrary points $L',L''\in F_L$ the map
\begin{equation}\label{eqn:CHF}
\xymatrix{\CH_0(F_L)_{\rm hom}\cong \CH_1(\tilde X)_{\rm hom}\ar@{->>}[r]^-{\gamma_\ast}&\CH_1(X)_{\rm hom}}
\end{equation} sends $[L']-[L'']\in \CH_0(F_L)_{\rm hom}$ 
to the class $[L']-[L'']\in \CH_1(X)_{\rm hom}$. In other words,
$$\xymatrix@C=28pt{\CH_0(F_L)_{\rm hom}\ar[r]^-{k_*\circ\tau_0^\ast}& \CH_1(\tilde X)_{\rm hom}\ar[r]^-{\gamma_*} &\CH_1(X)_{\rm hom}}$$
is indeed the map induced by the Fano correspondence.

Next recall from  the proof in Section \ref{sec:primanti} the standard fact that we have  $[L']+[L'']=[X\cap\PP^2]-[L]\equiv{\rm const}$ in $\CH_1(X)$
for any
$L'\in F_L$ and its image $L''\coloneqq\iota(L')\in F_L$ under the covering involution
of $\pi\colon F_L\to D_L$. Combining this with Lemma \ref{lem:ChF-}, we find that
the map (\ref{eqn:CHF}) factors through a surjection $\CH_0(F_L)_{\rm hom}^-\twoheadrightarrow \CH_1(X)_{\rm hom}$. However, as by virtue of Lemma \ref{lem:2alpha} we know that $\gamma^\ast\gamma_\ast(\alpha)=2\alpha$ for all $\alpha\in \CH_0(F_L)^-$ and since $\CH_0(F_L)$ is torsion free, the map is in fact
bijective.

\subsection{}\label{sec:ShenBP} The approaches to Theorem \ref{thm:mainChow} here and in \cite[Thm.\ 4.7]{Shen0} are different. Shen uses a family of correspondences of degree five for the Fano surfaces $F_L$, similar to the construction in \cite[Sec.\ 3]{Shen} for higher degree rational curves. The action of this correspondence is the composition of the Fano correspondence and its dual, which he identifies with the action $\iota^\ast-{\rm id}$. Our arguments instead rely on techniques due to Murre for cubic threefolds \cite{Murre} and, at least at first glance, there does not seem to be any link between the two.

\subsection{} Consider the general situation of \ref{sec:FDgen}, i.e.\ $\pi\colon F\twoheadrightarrow D$ is a finite morphism of degree two between surfaces with vanishing irregularity
and $\iota$ is the covering involution with its graph
$\Gamma_\iota\subset F\times F$.
Then the rational Chow motive $\hh(F)\in{\rm Mot}(\CC)$ decomposes
into $$\hh(F)\cong\hh(F)^+\oplus \hh(F)^-,$$
where $\hh(F)\coloneqq(F,[\Delta_F])$ and $\hh(F)^\pm\coloneqq(F,(1/2)([\Delta]\pm[\Gamma_\iota])$.
On the other hand, $$\hh(F)\cong\hh^0(F)\oplus\hh^2(F)\oplus\hh^4(F).$$
Note that due to $q(F)=0$, the odd parts $\hh^1(F)$ and $\hh^3(F)$ are both trivial
and recall that $\hh^0(F)\coloneq (F,p_0=[x\times F])\cong\LL^0$, 
$\hh^4(F)\coloneqq (F,p_4=[F\times x])\cong\LL^2$, and $\hh^2(F)\coloneqq
(F,[\Delta]-p_0-p_4)$.
The two decompositions are compatible in the sense that
$$\hh(F)^+\cong\hh^0(F)\oplus\hh^2(F)^+\oplus \hh^4(F)\text{ and }\hh(F)^-\cong\hh^2(F)^-.$$
Eventually, the choice of a polarization on $F$ allows one to write
$$\hh^2(F)\cong\hh^2(F)_{\rm pr}\oplus\LL
~ \text{ and }~\hh(F)^-=\hh^2(F)^-\cong\hh^2(F)^-_{\rm pr}\oplus \LL.$$

\subsection{}\label{sec:proofMot}  The Chow motive of a smooth cubic fourfold $X\subset\PP^5$ naturally decomposes
as $$\hh(X)\cong\bigoplus_{i=0}^4\LL^i\oplus \hh^4(X)_{\rm pr}.$$
The first summand corresponds to the image of $H^\ast(\PP^5,\QQ)\to H^\ast(X,\QQ)$ while the cohomology of the second is $H^4(X,\QQ)_{\rm pr}$. Recall that due to results of Bloch--Srinivas, cf.\ \cite[Prop.\ 22.27]{VoisinHodge}, and Murre \cite{Murre2}, the cycle class map $\CH^2(X)\otimes\QQ\to H^4(X,\QQ)$ is injective for the rationally connected variety $X$. Together with the Hodge conjecture for cubic fourfolds \cite{VoisinGT}, this gives $$\CH(\hh^4(X)_{\rm pr})\otimes\QQ= (\CH_2(X)_{\rm pr}\oplus \CH_1(X)_{\rm hom})\otimes\QQ\cong
 H^{2,2}(X,\QQ)_{\rm pr}\oplus \CH_1(X)_{\rm hom}\otimes\QQ.$$




Now, the Fano correspondence induces a morphism 
\begin{equation}\label{eqn:natmorphh}
\hh^2(F_L)_{\rm pr}^-\to\hh^4(X)_{\rm pr}(1)
\end{equation}
in ${\rm Mot}(\CC)$. Taking Chow groups (with $\QQ$-coefficients)
gives  $$\CH_0(F_L)^-
\oplus\CH_1(F_L)^-_{\rm pr}\to \CH_1(X)_{\rm hom}\oplus\CH_2(X)_{\rm pr}.$$
The first component is nothing but the isomorphism $\CH_0(F_L)^-_{\rm hom}\congpf \CH_1(X)_{\rm hom}$ proved in \ref{sec:proofCH}. The second defines the map
$$\xymatrix{H^{1,1}(F_L,\QQ)_{\rm pr}^-\cong\CH_1(F_L)^-_{\rm pr}\otimes\QQ\ar[r]& \CH_2(X)_{\rm pr}\otimes\QQ\ar[r]^-\sim& H^{2,2}(X,\QQ)_{\rm pr},}$$ which according to Proposition \ref{prop:MainHodge} is an isomorphism.

To conclude the verification of the second part of Theorem \ref{thm:mainmot}, i.e.\ the isomorphism $\hh^2(F)^-_{\rm pr}\cong\hh^4(X)_{\rm pr}(1)$, it is enough to observe that (\ref{eqn:natmorphh}) induces isomorphisms between their Chow groups even after arbitrary base field extension. By Manin's identity principle,  combine \cite[Lem.\ 1]{GG} and \cite[Lem.\ 3.2]{Pedrini}, see also
\cite[Lem.\ 1.1]{HuyMotK3} and \cite[Lem.\ 4.3]{Vial}, this implies that (\ref{eqn:natmorphh}) is indeed an isomorphism. This last argument is rather standard and essentially identical to the 
proof in \cite{BP}.


\bigskip

\appendix{}
\section{}
\centerline{ \scshape by John Ottem}

\bigskip

Let $X$ be a smooth cubic fourfold and $L\subset X$ a line. In this appendix, we prove that the surface $F_L\subset F(X)$ of lines in $X$ meeting $L$ is simply connected. 

\begin{thm}\label{maintheorem}
The surface $F_L$ is simply connected. 
\end{thm}

\def\F{ F'}
Let $\F$ denote the blow-up of $F_L$ in the point $[L]\in F_L\subset F(X)$. The starting point of the proof is the fact that $\F$ embeds into the projective bundle
$\PP(\kt_X|_L)$ over $L=\PP^1$.
Taking $L$ general, we may assume that $\kt_X|_L\simeq \E$, where $\E$ is the following vector bundle on $\PP^1$:
$$\E=\ko_{\PP^1}^2\oplus \ko_{\PP^1}(1)\oplus \ko_{\PP^1}(2).$$

Write $\rho \colon \PP(\E) \to \PP^1$ for the bundle projection. Then the Picard group of $\PP(\E)$ is generated by $\rho^\ast\ko(1)$ and 
the relative tautological $\ko_\rho(1)$. We will be interested in the following two line bundles
$$\kl_1\coloneqq\rho^\ast\ko(3)\otimes\ko_\rho(2) \quad \mbox{ and }\quad \kl_2\coloneqq 
\rho^\ast\ko(3)\otimes\ko_\rho(3).$$

\def\V{E}
 The surface $\F\subset \PP(\E)$ is defined by a section of the vector bundle $\V$ appearing as the extension, cf.\ \cite[Prop.\ 2.3.10]{HuyK3}:
\begin{equation}\label{extension}
0\to \kl_2 \to \V \to \kl_1\to 0.
\end{equation} In particular, we have 
\begin{equation}\label{flag}
\F=V(s_2)\subset V(s_1)\subset \PP(\E),
\end{equation}where $s_1\in H^0(\PP(\E),\kl_1)$ and $s_2\in H^0(V(s_1),\kl_2|_{V(s_1)})$.  

We will prove Theorem \ref{maintheorem} by a Lefschetz-type argument. This is slightly delicate, because the line bundles $\kl_1$ and $\kl_2$ are not ample. However, as we will see, they still have just enough positivity to make the argument work.

By Corollary \ref{parameterspace} below, the surface $\F$ is deformation equivalent to a surface $S=V(s)\subset D$ for a general divisor $D\in |\kl_1|$ and a general section $s\in H^0(D,\kl_2|_D)$. 

It will be convenient to pick sections $x_0,x_1,y_0,y_1,y_2,y_3$ such that:

\begin{itemize}
    \item 
$H^0(\PP(\E), \rho^\ast\ko(1))=\langle x_0,x_1 \rangle$

\item $H^0(\PP(\E), \ko_\rho(1))=\langle y_0, y_1 \rangle$

\item $H^0(\PP(\E), \rho^\ast\ko(1)\otimes\ko_\rho(1))=\langle x_0y_0, x_0y_1, x_1y_0, x_1y_1, y_2 \rangle$

\item $H^0(\PP(\E), \rho^\ast\ko(2)\otimes\ko_\rho(1))=\langle x_0^2y_0, x_0^2y_1, x_0x_1y_0, x_0x_1y_1,x_0y_2,x_1^2y_0,x_1^2y_1,x_1y_2,y_3\rangle$.
\end{itemize}
\smallskip

In terms of these sections, a basis of $H^0(\PP(\E),\kl_1)$ is given by the following 25 monomials:
\begin{align}\label{monomials}x_0^i x_1^{3-i}y_0^2, \quad x_0^i x_1^{3-i}y_0y_1,\quad x_0^i x_1^{3-i}y_1^2,\quad  x_0^i x_1^{2-i}y_0y_2,\\ \nonumber \quad x_0^i x_1^{2-i}y_1y_2, \quad x_iy_0y_3,\quad  x_iy_1y_3, \quad x_i y_2^2, \quad  y_2y_3.\end{align}
Similarly, $H^0(\PP(\E),\kl_2)$ has a basis consisting of the 38 monomials:
\begin{align}\label{monomials2}
x_0^i x_1^{3-i}y_0^3, \,\ldots,\, 
x_0^i x_1^{3-i}y_1^3,\quad 
x_0^i x_1^{2-i}y_0^2y_2,\ldots,  x_0^i x_1^{2-i}y_1^2y_2,\\ \nonumber \quad 
x_i y_0^{j}y_1^{2-j}y_3, \quad x_iy_0y_3,\quad  x_iy_1y_3, 
y_0y_2y_3,\quad
y_1y_2y_3,\quad
y_2^3.\end{align}

\begin{lem}\label{lemmaD1D2} The base locus of both linear systems
$|\kl_1|$ and $|\kl_2|$ is the curve $$Z=\PP(\ko(2))= V(y_0,y_1,y_2)\subset \PP(\E).$$A general divisor in $|\kl_1|$ is non-singular; a general divisor of $|\kl_2|$ has multiplicity two along $Z$.
\end{lem}

\begin{proof}
Both base loci must be contained in the base locus of $\rho^\ast\ko(1)\otimes\ko_\rho(1)$ which clearly equals $Z$, as we see using the basis above. Conversely, every monomial in \eqref{monomials} and \eqref{monomials2} is divisible by either $y_0,y_1$ or $y_2$, so $Z$ is exactly the base locus.

For the second claim, one easily checks that $x_0^3y_0^2+x_1^3y_0y_1+x_0^2x_1y_1^2+y_2y_3=0$ defines a non-singular divisor in $\PP(\E)$. Similarly, $y_1y_2y_3$ has multiplicity two along $Z$ and every other monomial of \eqref{monomials2} vanishes to order at least two there.
\end{proof}

\begin{remark}
It is natural to wonder whether the surface $\F$ is actually a complete intersection of divisors in $|\kl_1|$ and $|\kl_2|$ on $\PP(\E)$ (or stronger, whether the sequence \eqref{extension} splits).  However, this is excluded by  Lemma \ref{lemmaD1D2}, as any such complete intersection will be singular along $Z$, but $\F$ is smooth.

In fact,  for any $D\in |\kl_1|$ the restriction $\kl_2|_{D}$ admits a section which is not a restriction from $\PP(\E)$. For example, for $D=V(s_1)$ one can take
the section defining $\F\subset V(s_1)$. Indeed,  this follows from the exact sequence
\begin{equation}\label{sheafsequence}
0\to \ko_\rho(1) \to \kl_2 \to \kl_2|_{V(s_1)}\to 0
\end{equation}
and $H^1(\PP(\E),\ko_\rho(1))\cong H^1(\PP^1,\ko(-2))\cong \CC$, and $H^1(\PP(\E),\kl_2)\cong H^1(\PP^1, S^2(\E^\vee)\otimes \ko_{\PP^1}(3))=0$.
\end{remark}

By the sequence \eqref{sheafsequence}, the line bundle $\kl_2|_{D}$ has a 37-dimensional space of global sections for every $D\in |\kl_1|$. This implies that the parameter space of surfaces of the form $V(s_2)\subset V(s_1)$ as in (\ref{flag}) is a projective bundle over a projective space, hence it is a smooth projective variety.

\begin{cor}\label{parameterspace}
The parameter space of surfaces $\F$ appearing as in a flag \eqref{flag} is irreducible.
\end{cor}
Let $X\coloneqq D\setminus Z$ and let
$$
\pi \colon X\to \PP^{36}
$$ be  the morphism defined by the linear system $|\kl_2|_{D}|$.

\begin{lem}
The morphism $\pi$ is generically injective, and contracts no divisor.
\end{lem}
\begin{proof}
The 38 sections in \eqref{monomials2} define a morphism
$$
\tau\colon\PP(\E) \setminus Z \to \PP^{37}.
$$

We claim that $\tau$ contracts only the surface $V(y_0,y_1)$ to a point and  is an embedding on the complement $\PP(\E)\setminus V(y_0,y_1)$. Note first that $\tau$ is a toric morphism. This implies that the locus of points $z\in \PP^{37}$ for which the fiber $\tau^{-1}(z)$ has positive dimension is a union of linear subvarieties of $\PP^{37}$. Given the 38 monomials above, it is easy (for instance, in Macaulay2) to compute the fibers $\tau^{-1}(e_0), \ldots, \tau^{-1}(e_{37})$ over the standard coordinate points. Doing this, we find that there is a single point with positive dimensional fiber, namely $\tau^{-1}( [0:\ldots:0:1]) = V(y_0,y_1)$; for the remaining points, the fiber $\tau^{-1}(e_k)$ is either empty or consists of a single point. By semicontinuity of fiber dimension, this implies that $\tau$ is injective on the complement $\PP(\E)\setminus V(y_0,y_1)$.

Now, note that the intersection $\Gamma=D\cap V(y_0,y_1)$ is one-dimensional. This follows by Bertini, because the base locus of $\kl_1$ is $Z$ and $D$ is assumed to be general. It follows that the restriction $\tau|_{D\setminus Z} \colon D\setminus Z\to \PP^{37}$ contracts at most finitely many curves to points. This morphism is defined by a subsystem of $|\kl_2|_{D}|$, so the same conclusion holds for $\pi$.
\end{proof}

\begin{lem}\label{D1simply} The generic 
$D\in |\kl_1|$ is a  smooth projective rational threefold and hence simply connected.
\end{lem}

\begin{proof}  The generic $D\in|\kl_1=\rho^\ast(\ko(3)\otimes\ko_\rho(2)|$ 
is smooth, see Lemma \ref{lemmaD1D2}, and the projection $D\to \PP^1$ is generically a quadric bundle over $\PP^1$. Hence, $D$ is rational.
\end{proof}

\begin{proof}[Proof of Theorem \ref{maintheorem}]
Let $H\subset \PP^{36}$ denote a general hyperplane. 
Since $\pi$ defines a small birational morphism, we may apply a theorem of Goresky--MacPherson \cite[p.\ 150--151]{GMref} (where with the notation there $\hat n =2$), and deduce that the natural map
$$
\pi_i(\pi^{-1}(H)) \to \pi_i(X)
$$is an isomorphism for $i< 2$. However, $\pi_1(X)=0$ by Lemma \ref{D1simply}, because $D$ is simply connected and $X$ is obtained from $D$ by removing a closed subset of real codimension four.

Let now $S\in |\kl_2|_{D}|$ be a general divisor. Then $S$ is non-singular and irreducible, because $F$ is of this form. From the previous paragraph, we know that $S\setminus Z$, which equals $\pi^{-1}(H)$ for a general $H$, is simply connected. As $\pi_1(S\setminus Z)$ surjects onto $\pi_1(S)$, we deduce that $S$ is simply connected as well.  Finally, we note that $\F$ is deformation equivalent to $S$, so $\F$ and hence $F$ is simply connected. This completes the proof of Theorem \ref{maintheorem}.
\end{proof}

\end{document}